\documentclass[12pt]{article}

\RequirePackage[colorlinks,citecolor=blue,urlcolor=blue,linkcolor=blue]{hyperref}
\hypersetup{
colorlinks = true,
citecolor=blue,
urlcolor=blue,
linkcolor=blue,
pdfauthor = {Justin Miles},
pdfkeywords = {lognormal distribution, Laplace transform, characteristic function, analytic continuation, Mellin transform, series approximation.},
pdftitle = {The Laplace transform of the lognormal distribution},
pdfpagemode = UseNone
}
\usepackage{graphicx,xspace,colortbl}
\usepackage{epstopdf}
\usepackage{amsmath,amsthm,amsfonts,amssymb}
\usepackage{color}
\usepackage{enumerate}
\usepackage{fancybox}
\usepackage{epsfig}
\usepackage{pdfsync}
\usepackage{comment}
\usepackage{float}
\usepackage{tikz}
\usetikzlibrary{decorations.markings}
\usepackage{booktabs}
\usepackage{subcaption}

\makeatletter
\def\set@curr@file#1{%
  \begingroup
    \escapechar\m@ne
    \xdef\@curr@file{\expandafter\string\csname #1\endcsname}%
  \endgroup
}
\def\quote@name#1{"\quote@@name#1\@gobble""}
\def\quote@@name#1"{#1\quote@@name}
\def\unquote@name#1{\quote@@name#1\@gobble"}
\makeatother

\newtheorem{thm}{Theorem}
\newtheorem{prop}{Proposition}
\newtheorem{cor}{Corollary}
\newtheorem{lemma}{Lemma}

\DeclareMathOperator\re{Re}
\DeclareMathOperator\im{Im}
\DeclareMathOperator\Arg{Arg}
\DeclareMathOperator\sgn{sgn}

\DeclareMathOperator\erfc{erfc}
\DeclareMathOperator\Res{Res}
\DeclareMathOperator\LN{LN}
\def\r{{\mathbb R}}
\def\C{{\mathbb C}}
\def\N{{\mathbb N}}
\def\i{{\textnormal i}}
\def\d{{\textnormal d}}
\def\spc{\phantom{sss}}

\title{The Laplace transform of the lognormal distribution}
\author{Justin Miles%
	\thanks{Dept. of Mathematics and Statistics, York University, 4700 Keele Street, Toronto, ON, M3J 1P3, Canada.\newline Email: justinm@mathstat.yorku.ca}}
\date{}

\begin{document}

\maketitle

\begin{abstract}
We study the analytical properties of the Laplace transform of the lognormal distribution. Two integral expressions 
for the analytic continuation of the Laplace transform of the lognormal distribution are provided, one of which takes the form of a Mellin-Barnes integral. As
a corollary, we obtain an integral expression for the characteristic function; we show that the integral expression derived by Leipnik in \cite{Leipnik1991} is
incorrect. We present two approximations for the Laplace transform of the lognormal distribution, both valid in $\C \setminus(-\infty,0]$. 
In the last section, we discuss how one may use our results to compute the density of a sum of independent lognormal random variables.
\end{abstract}

\vspace{0.25cm}

{\vskip 0.15cm}


\section{Introduction}\label{intro}

A positive random variable $X$ is said to have a lognormal distribution with parameters $\mu\in \r$ and $\sigma>0$, written
$X \sim \LN(\mu,\sigma^2)$, if it has probability density function given by
\begin{align*}
 f(x;\mu,\sigma)=\frac{1}{\sqrt{2\pi}\sigma x}\exp{\left[-\frac{(\ln{x}-\mu)^2}{2\sigma^2}\right]},\spc x>0.
\end{align*} 

The lognormal distribution has a wide range of applications in the natural sciences and fields like finance, actuarial science, economics and engineering. 
Integral transforms, such as the Laplace and Fourier transforms, of the lognormal 
distribution have received considerable attention in the literature for several decades. 
The Laplace transform of $X$, henceforth denoted by $\phi$, is defined by 
\begin{equation}\label{def_phi}
 \phi(z;\mu,\sigma) := \mathbb{E}\left[e^{-zX}\right]
                       = \int_{0}^{\infty} e^{-zx}f(x;\mu,\sigma) \d x,\spc\re{(z)}\geq0.
\end{equation}
The characteristic function of $X$, henceforth denoted by $\varphi$, is the restriction of $\phi$ to the imaginary axis: 
\begin{equation}\label{def_varphi}
 \varphi(t;\mu,\sigma):=\mathbb{E}\left[e^{\i tX}\right]=\phi(- \i t;\mu,\sigma),\phantom{a}t\in\mathbb{R}.
\end{equation}

Since these integral transforms have no
known closed form, there has been substantial effort to put forth viable approximation methods (see \cite{Asmussen2016} for a 
thorough overview and numerical comparison of several methods).
Some authors, such as Barouch and Kaufman \cite{Barouch1976}, Barakat \cite{Barakat1976}, Holgate \cite{Holgate1989}, and Leipnik \cite{Leipnik1991}, have proposed series 
representations for \eqref{def_varphi}. Others, including Gubner \cite{Gubner2006} and Tellambura
and Senaratne \cite{Tellambura2010}, have proposed numerical integration methods for computing \eqref{def_phi}. Gubner's numerical integration procedure reduces oscillations of the integrand by deforming the contour of integration.
Tellambura and Senaratne improved upon Gubner's method by deriving the steepest-descent contour and by providing two, related,
closed-form contours.

More recently, Asmussen et al. \cite{Asmussen2016} used a modified version of Laplace's method to derive an asymptotically
equivalent, closed-form approximation for \eqref{def_phi}. Moreover, Asmussen et al. \cite{Asmussen2016} constructed a 
Monte Carlo estimator and, based on this framework, Laub et al.
\cite{Laub2016} generalized the approach to approximate the Laplace transform of a finite sum of dependent lognormals.

There are several disadvantages with existing methods in the literature. Examples include:
\begin{itemize}
 \item The majority of methods are only valid, at most,
for arguments in the right half plane, $\{z\in\mathbb{C}:\re{(z)}\geq0\}$. As a result, one must exclude some efficient paths of integration when performing an inversion of the
Laplace transform.

\item  It appears that there are no convergent series representations in the literature that are valid on the entire 
domain of analyticity. Since $\phi$ is not analytic at the origin, the
taylor series representation centered at any point will have finite radius of convergence.
For example, the formal Taylor series of $\phi$, centered at the origin, is given by 

\begin{equation}\label{maclaurin_phi}
 \sum_{n=0}^{\infty} \frac{(-z)^n}{n!}e^{\mu n +\frac{\sigma^2}{2}n^2}.
\end{equation}

It is easy to see that the series (\ref{maclaurin_phi}) diverges for all $z\neq0$. 

\item In 1991, Leipnik \cite{Leipnik1991} presented the following expression for the characteristic function: Let $X \sim \LN(0,\sigma^2)$,
then, for $t>0$ and $0<k<1$, the characteristic function is given by

\begin{equation}\label{leipnik_int_expression}
\varphi(t;0,\sigma) \stackrel{?}{=} \frac{1}{2\pi} \int_{k-\i\infty}^{k+\i\infty} \sin{(\pi s)}\Gamma(s)e^{-(\ln{t}+\i\frac{\pi}{2})s+
            \frac{\sigma^2}{2}s^2} \d s.
\end{equation}

It has been reported that the right-hand side of \eqref{leipnik_int_expression}, and the 
subsequent series for $\varphi$ derived in \cite{Leipnik1991}, are unreliable in numerical computations (see \cite{Dufresne2008}, and \cite{Asmussen2016}). We claim that the 
result is incorrect. To see that \eqref{leipnik_int_expression} is incorrect, observe that the integrand is entire and that one may take $k\in\r$. After 
shifting the contour of integration to the left of the origin (taking $k<0$), it is easy to see that the expression in \eqref{leipnik_int_expression} is $O(t^{|k|})$, 
as $t\to0$. Hence, the expression converges to $0$ as $t\to0$, violating the fact that the characteristic function must converge to $1$ as $t\to0$. 

Leipnik obtains \eqref{leipnik_int_expression} by first deriving a functional differential equation, and then solving it using a method due to de Bruijn. In this method, the differential 
equation is transformed into a forward difference-differential equation and an ansatz solution is posed. It appears that Leipnik imposed an inconvenient condition on the 
ansatz; specifically, in equation (25) of \cite{Leipnik1991}, he imposed the condition $S(z-1)=-S(z)$ when he could have taken $S(z-1)=S(z)$. As a result, Leipnik searched 
for an 
anti-periodic solution for $S$ and ultimately obtained $S(z)=\sin{\pi z}$ rather than $S(z)=1$.
\end{itemize}

In this paper, we explore the analytic continuation of the Laplace transform of the lognormal distribution and present new, efficient, series  \emph{approximations} of 
$\phi$ that are valid on $\C\setminus(-\infty,0]$.
In Sections \ref{section2} and \ref{section3}, we provide two (integral) expressions of the analytic continuation to $\C\setminus(-\infty,0]$, one of which takes the form of a Mellin-Barnes 
integral. As a corollary, we obtain an integral expression for the characteristic function $\varphi$ (this expression is stated by Dufresne in \cite{Dufresne2008} without 
proof).
In the third section of this paper, we exploit the Mellin-Barnes integral expression and use knowledge of the gamma function to derive series approximations 
for $\phi$ that are valid for arguments in $\C\setminus(-\infty,0]$.

The first approximation we present in Section \ref{section4}
is a convergent series for which the error term is uniformly bounded on $\C\setminus(-\infty,0]$ by a constant that can be made arbitrarily small 
(by choice of a parameter). Furthermore, the approximation is asymptotic to $\phi$ as the magnitude of the argument decreases to zero. 
The second approximation we present is a sum which improves as the parameter $\sigma\to\infty$. The terms of the 
series/sum are composed of expressions involving error functions and/or Hermite polynomials. The approximations are used to compute $\phi$ for several real 
arguments and the results compared to the values obtained by way of numerical integration.

In the last section, we discuss how one may use the analytic continuation of $\phi$ to compute the density of a sum of independent lognormals via Laplace inversion. 
By deforming the contour of the Bromwich integral to a Hankel contour, we obtain a real integral with an integrand which decays 
exponentially. The result is an integral which is easily evaluated numerically.


\section{The analytic continuation of the Laplace transform of the lognormal distribution}\label{section2}

The integral definition of the function $\phi$, given by \eqref{def_phi}, is finite when $\re{(z)}\geq0$ and it is well known that it is analytic in the right
half plane $\C^+=\{ z\in\C: \re{(z)}>0\}$. It will be convenient for us to express this function as

\begin{align}\label{phi_integral1}
  \phi(z;\mu,\sigma)&= C(\mu,\sigma) \int_{0}^{\infty} \frac{1}{x}\exp{\left[-zx-\frac{1}{2\sigma^2}(\ln{x})^2+\frac{\mu}{\sigma^2}
                          \ln{x}\right]} \d x,
\end{align}
where $C(\mu,\sigma):=(2\pi\sigma^2)^{-1/2}\exp{\left(-\mu^2/2\sigma^2\right)}$. Noting that the integral in \eqref{phi_integral1} is finite for all 
$\mu\in\C$, we define 
\begin{align}\label{def_bigphi}
 \Phi(z,w;\sigma) &:= C(w,\sigma)\int_{0}^{\infty} \frac{1}{x}\exp{\left[-zx-\frac{1}{2\sigma^2}(\ln{x})^2+\frac{w}{\sigma^2}
                          \ln{x}\right]} \d x, \spc (z,w)\in \overline{\C^+}\times\C.
\end{align}
Since $\Phi(z,\mu;\sigma)=\phi(z;\mu,\sigma)$, the function $\Phi$ is an extension of $\phi$. Here $\overline{\C^+}$ denotes the closure of $\C^+$, 
and throughout this paper we take the logarithm to be complex with the principal branch. The main result of this section is given in the following theorem. It provides us with an expression for
$\phi(z;\mu,\sigma)$ which is analytic on $\C\setminus(-\infty,0]$.
\begin{thm}\label{thm1}
Let $\sigma>0$ and let $\Phi$ be defined by \eqref{def_bigphi}. Then the Laplace transform of $X \sim \LN(\mu,\sigma^2)$ is analytically
continued to $\C\setminus(-\infty,0]$ by the equation
\begin{align}\label{int_expression_phi_1}
\phi(z;\mu,\sigma) &= \Phi(1,\mu+\ln{z};\sigma).
\end{align}
\end{thm}

\begin{proof} Fix $\sigma>0$ and let $F(z,w,x)=x^{-1}\exp{[-zx-(\ln{x})^2/2\sigma^2 +w\ln{x}/\sigma^2]}$ so that
\begin{align*}
\Phi(z,w;\sigma) &= C(w,\sigma)\int_{0}^{\infty} F(z,w,x) \d x.
\end{align*}
The function $C(\cdot,\sigma)$ is entire, and, for each $z\in\C^+$, $F(z,\cdot \phantom{,},\cdot)$ is continuous on $\C\times(0,\infty)$, and, 
for each pair  
$(z,x)\in \C^+\times(0,\infty)$, $F(z,\cdot \phantom{,},x)$ is entire. Thus, for each $n\in\N$, and for each 
$z\in\C^+$, the function $\Phi_n(z,\cdot \phantom{,};\sigma)$ defined by
\begin{align*}
 \Phi_n(z,w;\sigma)&:=  C(w,\sigma) \int_{\frac{1}{n}}^{n} F(z,w,x) \d x,\spc w\in\C
\end{align*}
is entire. Since $\Phi_n(z,\cdot\phantom{,};\sigma) \to \Phi(z,\cdot\phantom{,};\sigma)$ uniformly on compact subsets of $\C$, the function 
$\Phi(z,\cdot\phantom{,};\sigma)$ is entire. To prove the theorem, we make the formal substitution $ct=x$ in \eqref{def_bigphi} which yields
\begin{align}\label{continuation_of_Phi}
\Phi(z,w;\sigma) &=  C(w,\sigma) \int_{0}^{\infty} \frac{1}{ct}\exp{\left[-zct-\frac{1}{2\sigma^2}(\ln{ct})^2+\frac{w}{\sigma^2}
                        \ln{ct}\right]} c \d t\nonumber\\
                 &=  C (w,\sigma)\exp{\left(-\frac{1}{2\sigma^2}(\ln{c})^2 + \frac{w}{\sigma^2}\ln{c}\right)}\int_{0}^{\infty} \frac{1}{t}
		      \exp{\left[-zct-\frac{1}{2\sigma^2}(\ln{t})^2+\frac{(w-\ln{c})}{\sigma^2}\ln{t}\right]} \d t\nonumber\\
                 &=\Phi(cz,w-\ln{c};\sigma),
\end{align}
where \eqref{continuation_of_Phi} holds provided $cz\in\C^+$. Setting $c=1/z$ we obtain 
\begin{align}
 \Phi(z,w;\sigma) &=\Phi(1,w+\ln{z};\sigma).
\end{align}
Therefore, since $\Phi(z,\cdot\phantom{,};\sigma)$ is an entire function for each $z\in\C^+$, setting $w=\mu$ yields an analytic continuation for the Laplace transform 
of $X$ defined by \eqref{def_phi}.
\end{proof}


\section{The analytic continuation of \texorpdfstring{$\phi$}{} as a Mellin-Barnes integral}\label{section3}

In this section we derive an alternate expression for $\phi$, the Laplace transform of
$X \sim \LN(\mu,\sigma^2)$, in the form of a Mellin-Barnes integral. As a consequence, we obtain the corresponding
expression for the characteristic function of the lognormal random variable $X$. As far as we are aware, there is no other explicit proof of 
the result in the literature. For convenience, we will often write $f(x)$ and $\phi(z)$ instead of $f(x;\mu,\sigma)$ and $\phi(z;\mu,\sigma)$
with the understanding that $\mu$ and $\sigma$ are the parameters of $X$.
\begin{thm}\label{thm2}
Let $X \sim \LN(\mu,\sigma^2)$ and $k>0$. Then the Laplace transform of $X$ has integral expression
\begin{align}\label{int_expression_phi_2}
\phi(z) &= \frac{1}{2\pi i} \int_{k-\i\infty}^{k+\i\infty} \Gamma(s)e^{-(\mu+\ln{z})s+\frac{\sigma^2}{2}s^2} \d s,\spc z\in\C\setminus(-\infty,0].
\end{align}
\end{thm}
\begin{proof} The Mellin transform of $\phi$, denoted by $M[\phi;\cdot\phantom{,}]$, is defined by
\begin{align*}
M[\phi;s] &= \int_{0}^{\infty} z^{s-1}\phi(z) \d z, \spc s=k+\i t.
\end{align*}
Using the definition of the Laplace transform, Fubini's theorem, and the fact that
\begin{align*}
\int_{0}^{\infty} z^{s-1}e^{-zx} \d z &= x^{-s}\Gamma(s), \spc \re{(s)}>0,
\end{align*}
we have 
\begin{align*}
M[\phi;s]  &= \int_{0}^{\infty} z^{s-1} \left( \int_{0}^{\infty} e^{-zx}f(x) \d x \right)\d z\\
           &= \int_{0}^{\infty} f(x) \left( \int_{0}^{\infty} z^{s-1}e^{-zx} \d z \right)\d x\\
           &= \Gamma(s) \int_{0}^{\infty} x^{-s}f(x) \d x\\
           &= \Gamma(s) e^{-\mu s+\frac{\sigma^2}{2}s^2},\spc \re(s)>0.
\end{align*}
This also shows that $z^{k-1}\phi(z)\in L^1(0,\infty)$ for $k>0$. Furthermore, $\phi$ 
is continuous on $(0,\infty)$ and so, by Mellin's inversion formula (\cite{Titchmarsh} Pg.46, Theorem 28), 
\begin{equation*}
\phi(z) = \frac{1}{2\pi \i} \int_{k-\i\infty}^{k+\i\infty} z^{-s}M[\phi;s] \d s
        = \frac{1}{2\pi \i} \int_{k-\i\infty}^{k+\i\infty} \Gamma(s)e^{-(\mu+\ln{z})s+				          \frac{\sigma^2}{2}s^2} \d s, \spc z\in(0,\infty).
\end{equation*}

We can extend this function to take arguments in $\C\setminus(-\infty,0]$, and, in fact, it is analytic on this set. Therefore, our new expression for $\phi$ must agree with the analytic continuation given in Section \ref{section2} by the uniqueness of analytic continuation.
\end{proof}
Since $\varphi(t)=\phi(-it)$, we have the following corollary to Theorem \ref{thm2}
\begin{cor}\label{cor1}
Let $X \sim \LN(\mu,\sigma^2)$ and $k>0$. Then the characteristic function of 
$X$ has integral expression
\begin{align}
\varphi(t) &= \frac{1}{2\pi \i} \int_{k-\i\infty}^{k+\i\infty} \Gamma(s)e^{-(\mu+\ln{|t|}-\sgn{(t)}\i\frac{\pi}{2})s+
            \frac{\sigma^2}{2}s^2} \d s, \spc t\in\r\setminus\{0\}.
\end{align}
\end{cor}


\section{Series approximations and numerical computation of \texorpdfstring{$\phi$}{}}\label{section4}

In the first subsection we present series approximations which may be used to compute $\phi$ on $\C\setminus(-\infty,0]$. In the second subsection, we present numerical results using the series approximations and compare the error using numerical integration as a benchmark.

\subsection{Series approximations}\label{section4.1}

The following theorem introduces a convergent series that approximates $\phi$ with an error that can be made arbitrarily small. Note that the approximation bears some resemblance to the results of Barouch and Kaufman \cite{Barouch1976} who investigated series approximations of the characteristic function. 
\begin{thm}\label{thm3}
 Let $X \sim \LN(\mu,\sigma^2)$ and $\alpha\geq1$. Then the Laplace transform of $X$ has expression
 \begin{equation}\label{series_approx_1}
  \phi(z) = \sum_{n=0}^{\infty}\frac{(-z)^n}{n!}e^{\mu n+\frac{\sigma^2}{2}n^2}\cdot\frac{1}{2}\erfc{\left(\frac{\mu+\ln{(z/\alpha)}+\sigma^2n}{\sqrt{2}\sigma}\right)}\\
               +O(e^{-\alpha}), \spc z\in\C\setminus(-\infty,0].
 \end{equation}
Furthermore,
\begin{equation}\label{series_approx_1_asymptotic}
  \phi(z) \sim \sum_{n=0}^{\infty}\frac{(-z)^n}{n!}e^{\mu n+\frac{\sigma^2}{2}n^2}\cdot\frac{1}{2}\erfc{\left(\frac{\mu+\ln{(z/\alpha)}+\sigma^2n}{\sqrt{2}\sigma}\right)},
  \text{ as }z\to0.
\end{equation}
The function $\erfc$ is the complimentary error function.
\end{thm}
\begin{proof}
Let $\alpha\geq1$. For $\re{(s)}>0$, we may write $\Gamma(s)=\gamma(s,\alpha)+\Gamma(s,\alpha)$ where $\gamma(\cdot,\alpha)$ and $\Gamma(\cdot,\alpha)$ are the upper and lower incomplete gamma functions defined by 
\begin{equation*}
\gamma(s,\alpha) := \int_{0}^{\alpha} t^{s-1}e^{-t}\d t, \text{ and } \Gamma(s,\alpha) := \int_{\alpha}^{\infty} t^{s-1}e^{-t}\d t.
\end{equation*}
Substituting this sum into \eqref{int_expression_phi_2}, and replacing $\gamma(\cdot,\alpha)$ with the power series expansion
\begin{equation*}
 \gamma(s,\alpha)= \sum_{n=0}^{\infty} \frac{(-1)^n}{n!}\frac{\alpha^{s+n}}{(s+n)},
\end{equation*}
we obtain
\begin{equation*}
\phi(z) = \sum_{n=0}^{\infty} \frac{(-\alpha)^n}{n!} \frac{1}{2\pi \i} \int_{k-\i\infty}^{k+\i     			\infty} \frac{1}{(s+n)}e^{[\ln{\alpha}-(\mu+\ln{z})]s+\frac{\sigma^2}{2}s^2}\d s
        + \frac{1}{2\pi \i} \int_{k-\i\infty}^{k+\i\infty} \Gamma(s,\alpha)e^{-(\mu+\ln{z})s+					\frac{\sigma^2}{2}s^2}\d s,
\end{equation*}
where $k>0$ (when necessary). The interchange of summation and integration in the first term is justified by Fubini's theorem and the fact that the integral can be bounded
by a Gaussian integral, independent of $n$. To complete the proof, we need to:
\begin{enumerate}
  \item[i)] compute $\frac{1}{2\pi \i} \int_{k-\i\infty}^{k+\i\infty} \frac{1}{(s+n)}e^{[\ln{\alpha}-(\mu+\ln{z})]s+\frac{\sigma^2}{2}s^2}\d s,\spc n\in\N\cup\{0\}$, and
  \item[ii)] bound  $\left|\frac{1}{2\pi \i} \int_{k-\i\infty}^{k+\i\infty} \Gamma(s,\alpha)e^{-(\mu+\ln{z})s+\frac{\sigma^2}{2}s^2}\d s\right|$.
\end{enumerate}
 
We compute the integral in i) using differentiation with respect to a parameter. Letting
\begin{equation*}
 F_n(w)=\frac{1}{2\pi \i} \int_{k-\i\infty}^{k+\i\infty} \frac{1}{(s+n)} e^{w(s+n)+\frac{\sigma^2}{2}s^2}\d s,
\end{equation*}
we have
\begin{equation*}
 F_n'(w) = \frac{1}{2\pi \i} \int_{k-\i\infty}^{k+\i\infty} e^{w(s+n)+\frac{\sigma^2}{2}s^2}\d s
         = \frac{1}{\sqrt{2\pi}\sigma} e^{wn-\frac{w^2}{2\sigma^2}},
\end{equation*}
and so, for $ w\in\r$,
\begin{equation*}
F_n(w) = F_n(-\infty)+\int_{-\infty}^{w} F_n'(y)\d y
       = 0+\frac{e^{\frac{\sigma^2}{2}n^2}}{\sqrt{2\pi}\sigma}\int_{-\infty}^{w} e^{-\frac{(y-				\sigma^2n)^2}{2\sigma^2}}\d y
       = \frac{e^{\frac{\sigma^2}{2}n^2}}{2}\erfc{\left(\frac{-w+\sigma^2n}{\sqrt{2}\sigma}					\right)}.
\end{equation*}
It can be shown that the interchange of integration and differentiation is justified, and $F_n(-\infty)=0$ using the dominated convergence theorem. Since the 
complementary error function is entire, the result extends to arguments in $\mathbb{C}$ by analytic continuation. Therefore
\begin{align*}
\frac{1}{2\pi \i} \int_{k-\i\infty}^{k+\i\infty} \frac{1}{(s+n)}e^{[\ln{\alpha}-(\mu+\ln{z})]s+\frac{\sigma^2}{2}s^2}\d s
&= e^{-[\ln{\alpha}-(\mu+\ln{z})]n}F_n\left(\ln{\alpha}-(\mu+\ln{z})\right)\\
&= \alpha^{-n}z^ne^{\mu n+\frac{\sigma^2}{2}n^2}\cdot\frac{1}{2} \erfc{\left(\frac{\mu+\ln{(z/\alpha)}+\sigma^2n}{\sqrt{2}\sigma}\right)}
\end{align*}

To bound the integral in ii), observe that the integrand is entire and we may choose any $k\in\r$ for the vertical contour. Choosing $k\leq1$, we have 
$\left|\Gamma(s,\alpha)\right|\leq e^{-\alpha}$ for $s=k+\i t,\phantom{a}t\in\r$, and so
\begin{align*}
\left|\frac{1}{2\pi \i} \int_{k-\i\infty}^{k+\i\infty} \Gamma(s,\alpha)e^{-(\mu+\ln{z})s+\frac{\sigma^2}{2}s^2}\d s\right|
& \leq  \frac{e^{-\alpha}}{2\pi} \int_{-\infty}^{\infty} e^{-(\mu+\ln|z|)k+\Arg{(z)}t+\frac{\sigma^2}{2}(k^2-t^2)}\d t\\
& \leq \frac{1}{\sqrt{2\pi}\sigma}e^{\frac{\pi^2}{2\sigma^2}+\frac{\sigma^2}{2}k^2-\alpha-k(\mu+\ln{|z|})}.
\end{align*}

We set $k=0$ to obtain the error term in \eqref{series_approx_1} and we choose $k$ to be negative to show \eqref{series_approx_1_asymptotic}.
\end{proof}

The following theorem presents an approximation of $\phi$ that improves as the parameter $\sigma$ increases.
\begin{thm}\label{thm4}
 Let $X \sim \LN(\mu,\sigma^2)$, and $M, N\in\N$. Then the Laplace transform of $X$ has expression
  \begin{align}\label{series_approx_2}
   \phi(z) &= \sum_{n=0}^{N} \frac{(-z)^n}{n!} e^{\mu n+\frac{\sigma^2}{2}n^2}\cdot\frac{1}{2}\erfc{\left(\frac{\mu+\ln{z}+\sigma^2n}{\sqrt{2}\sigma}\right)}\nonumber\\
    &+ \sum_{m=0}^{M} \frac{(-1)^ma_m}{\sqrt{2\pi}\sigma^{m+1}}e^{-\frac{(\mu+\ln{z})^2}{2\sigma^2}}H_m\left(-\frac{(\mu+\ln{z})}{\sigma}\right) +O(\sigma^{-M-2}).
 \end{align}
The function $\erfc$ is the complimentary error function, $H_m$ is the $m^{th}$ (probabilist's) Hermite polynomial defined by
 \begin{equation*}
  H_m(x):=(-1)^m e^{\frac{x^2}{2}} \frac{\d ^m}{\d x^m}e^{-\frac{x^2}{2}},
 \end{equation*}
 and the coefficients $a_m$ are defined by
 \begin{equation*}
  a_m = \frac{\Gamma^{(m+1)}(1)}{(m+1)!}+(-1)^{m+1}\cdot \sum_{j=1}^{N} \frac{(-1)^j}{j!}\frac{1}{j^{m+1}}.
 \end{equation*}
\end{thm}
\begin{proof}
Let $N\in\mathbb{N}$. Recall that the function $\Gamma$ has a simple pole at $s=-n,\phantom{a}n=0,1,2,\ldots$, with residue $\Res{(\Gamma,-n)}=(-1)^n/n!$.
We may remove the first $N+1$ poles of $\Gamma$ by writing
\begin{equation*}
 \Gamma(s)-\sum_{n=0}^{N}\frac{(-1)^n}{n!}\frac{1}{(s+n)}
\end{equation*}
to obtain a function, denoted $\gamma_N$, that is holomorphic on $\{s\in\mathbb{C}:|s|<N+1\}$. Thus, we write
 \begin{equation*}
  \Gamma(s)=\gamma_N(s)+\sum_{n=0}^{N} \frac{(-1)^n}{n!}\frac{1}{(s+n)},
 \end{equation*}
where, for $|s|<N+1$, we may write 
 \begin{equation*}
  \gamma_N(s)=\sum_{m=0}^{\infty} a_m s^m,
 \end{equation*}
 with $a_m=\gamma_N^{(m)}(0)/m!$. Note that as $\sigma\to\infty$, the mass of the integrand in the integral \eqref{int_expression_phi_2} is 
 increasingly supported on the set $\{k+\i t\in\mathbb{C}: t\in(-N-1,N+1)\}$. Thus, we choose $M\in\mathbb{N}$ and write
\begin{align}
  \phi(z) &= \frac{1}{2\pi \i} \int_{k-\i\infty}^{k+\i\infty}\left( \gamma_N(s)+\sum_{n=0}^{N} \frac{(-1)^n}{n!}\frac{1}{(s+n)}\right)e^{-(\mu+\ln{z})s+\frac{\sigma^2}{2}s^2}\d s\nonumber\\
             &= \sum_{n=0}^{N} \frac{(-1)^n}{n!}\cdot \frac{1}{2\pi \i} \int_{k-\i\infty}^{k+\i\infty}\frac{1}{(s+n)}e^{-(\mu+\ln{z})s+\frac{\sigma^2}{2}s^2} \d s \nonumber\\
             &+ \sum_{m=0}^{M} a_m\cdot \frac{1}{2\pi \i} \int_{-\i\infty}^{+\i\infty}s^m e^{-(\mu+\ln{z})s+\frac{\sigma^2}{2}s^2}\d s + R_M(z),
\end{align}
where,
 \begin{equation*}
  R_M(z)=\frac{1}{2\pi \i} \int_{-\i\infty}^{+\i\infty}\left( \gamma_N(s)-\sum_{m=0}^{M} a_ms^m\right)e^{-(\mu+\ln{z})s+\frac{\sigma^2}{2}s^2}\d s.
 \end{equation*}
 
To complete the proof we need to:
 \begin{enumerate}
  \item[i)] compute $\frac{1}{2\pi \i} \int_{k-\i\infty}^{k+\i\infty}\frac{1}{(s+n)}e^{-(\mu+\ln{z})s+\frac{\sigma^2}{2}s^2}\d s, \spc n\in\{0,1,\ldots,N\}$,
  \item[ii)] compute $a_m,\phantom{a} m\in\{0,1,\ldots,M\}$,
  \item[iii)] compute $ \frac{1}{2\pi \i} \int_{k-\i\infty}^{k+\i\infty}s^m e^{-(\mu+\ln{z})s+\frac{\sigma^2}{2}s^2}\d s,\spc m\in\{0,1,\ldots,M\}$, and
  \item[iv)] bound $|R_M(z)|$
 \end{enumerate}
 
The integral in i) was computed in the proof of Theorem 3. To determine ii), we need to compute $\gamma_N^{(m)}(0)$. Note that, for $|s|<1$, we may write 
\begin{equation*}
 \Gamma(s)-\frac{1}{s}=\sum_{j=0}^{\infty} b_js^j,
\end{equation*}
where $b_j=\Gamma^{(j+1)}(1)/(j+1)!$. So, for $|s|<1$, we may write
\begin{equation*}
 \gamma_N(s)= \sum_{j=0}^{\infty} b_js^j - \sum_{n=1}^{N} \frac{(-1)^n}{n!}\frac{1}{(s+n)}.
\end{equation*}
Differentiating, we have
\begin{equation*}
 \gamma_N^{(m)}(s)= \sum_{j=m}^{\infty} (j)_mb_js^{j-m}+(-1)^{m+1}m!\sum_{n=1}^{N} \frac{(-1)^n}{n!}\frac{1}{(s+n)^{m+1}}
\end{equation*}
and
\begin{equation*}
 \gamma_N^{(m)}(0)= m!b_m+(-1)^{(m+1)}m!\sum_{n=1}^{N}\frac{(-1)^n}{n!}\frac{1}{n^{m+1}}.
\end{equation*}
Therefore,
\begin{equation*}
 a_m = \frac{\gamma_N^{(m)}(0)}{m!}
     = b_m+(-1)^{(m+1)}\sum_{n=1}^{N}\frac{(-1)^n}{n!}\frac{1}{n^{m+1}}.
\end{equation*}
To compute iii), we let 
\begin{equation*}
 G(w) = \frac{1}{2\pi \i} \int_{k-\i\infty}^{k+\i\infty} e^{ws+\frac{\sigma^2}{2}s^2}\d s
      = \frac{1}{\sqrt{2\pi}\sigma}e^{-\frac{w^2}{2\sigma^2}}.
\end{equation*}
Then
\begin{equation*}
 G^{(m)}(w) =\frac{\d ^m}{\d w^m} \left(\frac{1}{\sqrt{2\pi}\sigma}e^{-\frac{w^2}{2\sigma^2}}\right)
            =\frac{1}{\sqrt{2\pi}\sigma}\cdot \frac{(-1)}{\sigma^m}^me^{-\frac{w^2}{2\sigma^2}}H_m\left(\frac{w}{\sigma}\right),
\end{equation*}
and therefore,
\begin{equation*}
 \frac{1}{2\pi \i} \int_{k-\i\infty}^{k+\i\infty}s^m e^{-(\mu+\ln{z})s+\frac{\sigma^2}{2}s^2}\d s
 = G^{(m)}\left(-(\mu+\ln{z})\right)
 = \frac{(-1)^m}{\sqrt{2\pi}\sigma^{m+1}}e^{-\frac{(\mu+\ln{z})^2}{2\sigma^2}}H_m\left(\frac{-(\mu+\ln{z})}{\sigma}\right),
\end{equation*}
where, again, it can be shown that the interchange of integration and differentiation is justified. Finally, to show iv), we note that 
\begin{equation*}
 \left|\gamma_N(s)-\sum_{m=0}^{M}a_m s^m\right| \leq C|s|^{M+1},\spc s\in \i\r,
\end{equation*}
for some $C>0$. To see this, observe that, for $|s|<N$, we have 
\begin{equation*}
 \left|\gamma_N(s)-\sum_{m=0}^{M}a_m s^m\right|=\left|\sum_{m=M+1}^{\infty} a_m s^m\right|=\left|\sum_{m=0}^{\infty} a_{M+1+m} s^m\right|\cdot|s|^{M+1}\leq C_1|s|^{M+1}.
\end{equation*}
We also have
\begin{equation*}
 \left|\gamma_N(s)\right|= \left|\Gamma(s)-\frac{1}{s} - \sum_{n=1}^{N}\frac{(-1)^n}{n!}\frac{1}{(s+n)}\right|
 \leq \left|\Gamma(s)-\frac{1}{s}\right| + \left|\sum_{n=1}^{N}\frac{(-1)^n}{n!}\frac{1}{(s+n)}\right|
 \leq C_2,\spc s\in \i\r,
\end{equation*}
so that, for $|s|\geq N$, we have
\begin{equation*}
 \left|\gamma_N(s)-\sum_{m=0}^{M}a_m s^m\right|\leq  \left|\gamma_N(s)\right|+\left|\sum_{m=0}^{M}a_m s^m\right| 
 \leq C_2 + C_3|s|^M
 \leq C_4|s|^{M+1}.
\end{equation*}
Thus,
\begin{equation*}
 |R_M(z)| \leq \frac{C}{2\pi} \int_{-\infty}^{\infty} |t|^{M+1} e^{\Arg{(z)}t-\frac{\sigma^2}{2}t^2}\phantom{,}dt
 = \frac{C}{2\pi} \int_{-\infty}^{\infty} \left|\frac{x}{\sigma}\right|^{M+1} e^{\frac{\Arg{(z)}}{\sigma}x-\frac{x^2}{2}}\phantom{,}\frac{1}{\sigma} dx
 \leq C'\sigma^{-M-2}.
\end{equation*}
\end{proof}

\subsection{Numerical examples}\label{section4.2}

We can compute $\phi(z)$, $z\in\C\setminus(-\infty,0]$, via numerical integration using either of the relations \eqref{int_expression_phi_1} or \eqref{int_expression_phi_2} given by 
Theorems \ref{thm1} or \ref{thm2}, respectively. If we choose to use the former, then we need to compute $\Phi(1,\mu + \ln{z}; \sigma)$ as
defined by \eqref{def_bigphi}. For simplicity, we will discuss the computation of $\Phi(1,a+\i b;\sigma)$, for $a,b\in\mathbb{R}$. 

Making the substitution $x\to e^{x}$ we have
\begin{equation*}
\Phi(1,a+\i b;\sigma)= C(a+\i b,\sigma)\int_{-\infty}^{\infty} \exp{\left[ -e^x - \frac{x^2}{2\sigma^2} +\frac{(a+\i b)x}{\sigma^2}\right]} \d x.
\end{equation*}
We can write this integral in the form
\begin{equation}\label{filon_integral}
 \int_{-\infty}^{\infty} g(x)e^{\i tx} \d x,
\end{equation}
where
\begin{equation*}
g(x) = \exp{\left[ -e^x - \frac{x^2}{2\sigma^2} +\frac{ax}{\sigma^2}\right]}, \text{ and } t=\frac{b}{\sigma^2}.
\end{equation*}

The integral \eqref{filon_integral} can be computed numerically using Filon's quadrature method \cite{Filon1928}. First, we determine an interval, $[x_0,x_{2N}]$, which supports most of the integrand's
mass and create a mesh consisting of $2N+1$ points, $x_j,\phantom{a}j=0,\ldots,2N$. The integral is then written as a 
sum of $N$ integrals over $[x_{2j},x_{2j+2}]$, $j=0,\ldots, N-1$:
\begin{equation*}
\int_{-\infty}^{\infty} g(x)e^{\i tx} \d x \approx \int_{x_0}^{x_{2N}} g(x)e^{\i tx} \d x
                                                  = \sum_{j=0}^{N-1} \int_{x_{2j}}^{x_{2j+2}} g(x)e^{\i tx} \d x.
\end{equation*}

On each subinterval $[x_{2j},x_{2j+2}]$, we approximate $g(x)$ with a second order Lagrange interpolating polynomial using the
data points $(x_{2j},g_{2j}),(x_{2j+1},g_{2j+1})$, and $(x_{2j+2},g_{2j+2})$, where $g_j:= g(x_j)$. Thus, with 
$g(x)\approx c_0^{(j)} + c_1^{(j)}x+c_2^{(j)}x^2$ on $[x_{2j},x_{2j+2}]$, we have 
\begin{equation}\label{filon_integral2}
 \int_{-\infty}^{\infty} g(x)e^{\i tx} \d x \approx \sum_{j=0}^{N-1} \int_{x_{2j}}^{x_{2j+2}} \left(c_0^{(j)} + c_1^{(j)}x+c_2^{(j)}x^2\right)e^{\i tx} \d x.
\end{equation}

The integrals on the right hand side of \eqref{filon_integral2} can be computed explicitly. With an appropriate interval of integration, 
$[x_0,x_{2N}]$, and $N$ sufficiently large, an accurate approximation of the integral \eqref{filon_integral} is obtained.

Theorem \ref{thm3} was used to numerically compute $\phi(z)$ for several real values of $z$; Table \ref{table1} shows the results 
corresponding to $\sigma= 0.0625,0.25,0.75,$ and $1$ with $\mu=0$. In each case, the expression in \eqref{series_approx_1} was truncated to 41 terms and evaluated using $\alpha=10$.
Table \ref{table2} displays the absolute difference (labeled AD) between $\phi(z)$ computed using \eqref{series_approx_1} and the value of $\phi(z)$ computed by way of numerical integration.

Theorem \ref{thm4} was used to numerically compute $\phi(z)$ for several real values of $z$; Table \ref{table3} shows the results 
corresponding to $\sigma= 1, 1.5, 2,$ and $2.5$ with $\mu=0$. In each case, \eqref{series_approx_2} was used with $N=5$, and $M=10$. Table \ref{table4} displays the absolute difference (labeled AD) between $\phi(z)$ computed using \eqref{series_approx_2} and the value of $\phi(z)$ computed by way of numerical integration.

\begin{table}
\caption{The function $\phi$ computed using \eqref{series_approx_1}, truncated to 41 terms, with $\alpha=10$.} 
\label{table1} 
\begin{center}
\def\arraystretch{1.5}
 \begin{tabular}{| c |c | c | c |  c |}
  \toprule[1.5pt]
  \phantom{z} & $\sigma=0.0625$&$\sigma=0.25$&$\sigma=0.75$&$\sigma=1$\\
  \hline
    $z$ & $\phi(z)$   & $\phi(z)$ & $\phi(z)$  & $\phi(z)$  \\
  \bottomrule[1.25pt]
0.5&0.60624&0.60196&0.57541&0.56171\\\hline
1&0.36788&0.36804&0.37469&0.38176\\\hline
1.5&0.22346&0.22825&0.26086&0.2807\\\hline
2&0.13586&0.14342&0.18984&0.21631\\\hline
3&0.050369&0.058656&0.10995&0.14025\\\hline
5&0.0070017&0.011065&0.045898&0.072028\\\hline
10&3.9289e-05&0.00028124&0.0096044&0.022991\\
\bottomrule[1.25pt]
 \end{tabular}
\end{center}
\end{table}

\begin{table}
\caption{absolute difference between $\phi(z)$ computed using \eqref{series_approx_1} and $\phi(z)$ computed using numerical integration.}
\label{table2}
\begin{center}
\def\arraystretch{1.5}
 \begin{tabular}{| c |c | c | c |  c |}
  \toprule[1.5pt]
  \phantom{z} & $\sigma=0.0625$&$\sigma=0.25$&$\sigma=0.75$&$\sigma=1$\\
  \hline
    $z$ & AD   & AD & AD   & AD   \\
  \bottomrule[1.25pt]
0.5&6.572520e-14&6.661338e-14&5.155796e-10&1.478849e-08\\\hline
1&1.110223e-16&4.013456e-14&1.456569e-08&9.738506e-08\\\hline
1.5&4.440892e-16&2.525757e-14&6.936278e-08&2.349823e-07\\\hline
2&2.775558e-17&1.468270e-14&1.760354e-07&3.975210e-07\\\hline
3&1.942890e-16&2.212219e-11&5.105122e-07&7.251790e-07\\\hline
5&1.756408e-15&6.980607e-08&1.292328e-06&1.225385e-06\\\hline
10&1.450124e-05&6.055084e-06&2.185399e-06&1.648996e-06\\
\bottomrule[1.25pt]
 \end{tabular}
\end{center}
\end{table}

\begin{table}
\caption{The function $\phi$ computed using \eqref{series_approx_2} with $N=5$, and $M=10$.}  
\label{table3}
\begin{center}
\def\arraystretch{1.5}
 \begin{tabular}{| c |c | c | c |  c |}
  \toprule[1.5pt]
  \phantom{z} & $\sigma=1$&$\sigma=1.5$&$\sigma=2$&$\sigma=2.5$\\
  \hline
    $z$ & $\phi(z)$   & $\phi(z)$ & $\phi(z)$  & $\phi(z)$  \\
  \bottomrule[1.25pt]
0.5&0.56169&0.54186&0.53012&0.523\\\hline
1&0.38175&0.39772&0.41216&0.42396\\\hline
1.5&0.28073&0.31674&0.34538&0.36751\\\hline
2&0.21634&0.26336&0.30039&0.32893\\\hline
3&0.14024&0.19613&0.24163&0.27744\\\hline
5&0.072008&0.12725&0.17708&0.21855\\\hline
10&0.023002&0.062944&0.10844&0.15117\\
\bottomrule[1.25pt]
 \end{tabular}
\end{center}
\end{table}

\begin{table}
\caption{absolute difference between $\phi(z)$ computed using \eqref{series_approx_2} and $\phi(z)$ computed using numerical integration.} 
\label{table4}
\begin{center}
\def\arraystretch{1.5}
 \begin{tabular}{| c |c | c | c |  c |}
  \toprule[1.5pt]
  \phantom{z} & $\sigma=1$&$\sigma=1.5$&$\sigma=2$&$\sigma=2.5$\\
  \hline
    $z$ & AD   & AD & AD   & AD   \\
  \bottomrule[1.25pt]
0.5&3.503349e-05&6.444704e-07&2.668369e-08&3.269640e-09\\\hline
1&1.716174e-05&3.009667e-08&1.325727e-09&9.603728e-10\\\hline
1.5&1.196279e-04&8.780255e-07&2.567335e-08&1.195540e-09\\\hline
2&1.371616e-04&1.352950e-06&4.380613e-08&2.832041e-09\\\hline
3&6.300887e-05&1.180623e-06&5.771800e-08&4.875454e-09\\\hline
5&2.828704e-04&7.533040e-07&4.059893e-08&6.225437e-09\\\hline
10&4.467548e-04&3.262143e-06&4.479331e-08&4.615517e-09\\
\bottomrule[1.25pt]
 \end{tabular}
\end{center}
\end{table}


\section{The density of a sum of independent lognormal random variables}
We have introduced two integral expressions which analytically continue the Laplace transform of 
$X \sim \LN(\mu,\sigma^2)$ to $\C\setminus(-\infty,0]$. We have also provided series approximations which may be used for numerical computations. In the last section of this paper, we consider an application which utilizes the analytic continuation of the Laplace transform of $X$.

In this section we discuss a method to numerically compute the density of a sum of independent lognormal random variables. In
this procedure, we obtain the density function by inverting the Laplace transform of the sum. 
Using the analytic continuation of the Laplace transform, we may deform the contour of the Bromwich integral 
into a Hankel contour and obtain an integral for which the integrand decays exponentially.\par
\begin{prop}\label{prop1}
Let $X_j \sim \LN(\mu_j,\sigma_j^2)$, $j=1,\ldots,n$, be independent and $X=\sum_{j=1}^{n}X_j$. Then X has density
\begin{equation}\label{density_integral}
f_X(x)= -\frac{1}{\pi} \int_{0}^{\infty} \im{[\phi(-t+\i \cdot0)]}e^{-tx} \d t, \spc x>0.
\end{equation}
Here $\phi$ and $f_X$ denote the Laplace transform and probability density function of $X$, respectively, and 
$\phi(-t +\i\cdot0)=\lim_{\epsilon\to0^+} \phi(-t+\i\epsilon)$.
\end{prop}

We will use the following lemma in the proof of Proposition \ref{prop1}

\begin{lemma}\label{lemma1}
Let $\phi_j$ denote the Laplace transform of $X_j \sim \LN(\mu_j,\sigma_j^2)$. For every
$k>0$, $\phi_j(z)= O_k(|z|^{-k})$ as $|z|\to\infty$, $z\in\C\setminus(-\infty,0]$. Consequently, for every $k>0$, $\phi(z)= O_k(|z|^{-k})$ as $|z|\to\infty$, 
$z\in\C\setminus(-\infty,0]$.
\end{lemma}

\begin{proof}[Proof of Lemma 1] Let $k>0$ and $z\in\mathbb{C}\setminus (-\infty,0]$. By Theorem \ref{thm2}, we have
\begin{equation*}
\phi_{j}(z)= \frac{1}{2\pi} \int_{-\infty}^{\infty} z^{-(k+\i t)}\Gamma(k+\i t)e^{-\mu_j (k+\i t) +\frac{\sigma_{j}^2}{2}(k+\i t)^2}\d t,
\end{equation*}
and thus,
\begin{equation*}
|\phi_{j}(z)| \leq \frac{1}{2\pi} \int_{-\infty}^{\infty} |z|^{-k}e^{\pi |t|}\Gamma(k)e^{-\mu_j k+\frac{\sigma_j^2}{2}(k^2-t^2)}\d t= M_{k,j}|z|^{-k}.
\end{equation*}

To prove the second part of the lemma, let $k>0$ and take $r=k/n$. Then, by the independence of the $X_j$'s and the first 
part of the lemma,
\begin{equation*}
|\phi(z)| = \prod_{j=1}^{n} |\phi_j(z)|
             \leq \prod_{j=1}^{n} M_{r,j}|z|^{-r}
             = M_k|z|^{-k}
\end{equation*}
where, $M_k=\prod_{j=1}^{n} M_{r,j}$.
\end{proof}

\begin{proof}[Proof of Proposition 5] The density function of $X$, obtained by the inverse Laplace transform, is given by the Bromwich integral 
\begin{equation}\label{bromwich_integral}
f_X(x)= \frac{1}{2\pi \i} \int_{c-\i\infty}^{c+\i\infty} \phi(z)e^{zx} \d z,
\end{equation}
for any $c>0$. Using the analytic continuation of $\phi$, the integrand of \eqref{bromwich_integral} is analytic on $\C\setminus(-\infty,0]$ and we can deform the contour to the contour 
$\Gamma_R=\gamma_1^{(-)}+\gamma_2^{(-)}+H_R+\gamma_2^{(+)}+\gamma_1^{(+)}$, shown in Figure \ref{fig1}, for any $R>0$. The density function is now given by
\begin{equation}
f_X(x)= \frac{1}{2\pi \i} \int_{\Gamma_R}\phi(z)e^{zx} \d z.
\end{equation}

\begin{figure}
\centering
\begin{tikzpicture}
[decoration={markings,
mark=at position 0.5cm with {\arrow[line width=1pt]{>}},
mark=at position 3.25cm with {\arrow[line width=1pt]{>}},
mark=at position 6.25cm with {\arrow[line width=1pt]{>}},
mark=at position 8.25cm with {\arrow[line width=1pt]{>}},
mark=at position 10.5cm with {\arrow[line width=1pt]{>}},
mark=at position 13.5cm with {\arrow[line width=1pt]{>}},
mark=at position 16.25cm with {\arrow[line width=1pt]{>}}
}
]
\draw[help lines,->] (-4,0) -- (4,0) coordinate (xaxis);
\draw[help lines,->] (0,-4) -- (0,4) coordinate (yaxis);

\path[draw,line width=0.8pt,postaction=decorate] (1,-4)--(1,-3) node[right]{$c-iR$} to[bend left] (-2,-0.5) -- (-0.5,-0.5) arc(-135:135:0.707)--(-2,0.5) to[bend left] (1,3) node[right]{$c+iR$}--(1,4);

\foreach \Point in {(1,-3), (1,3), (-2,-0.5), (-2,0.5)}{
    \node at \Point {\textbullet};
    }
\node[below] at (xaxis) {$\re{(z)}$};
\node[left] at (yaxis) {$\im{(z)}$};
\node[below right] {$0$};
\node at (-1,-2.75) {$\gamma_2^{(-)}$};
\node at (-1,2.75) {$\gamma_2^{(+)}$};
\node at (0.5,-3.6) {$\gamma_1^{(-)}$};
\node at (0.5,3.6) {$\gamma_1^{(+)}$};
\node at (1,0.3) {$H_R$};
\end{tikzpicture}
\caption{ The contour  $\Gamma_R=\gamma_1^{(-)}+\gamma_2^{(-)}+H_R+\gamma_2^{(+)}+\gamma_1^{(+)}$}
\label{fig1}
\end{figure}

We will show that the contributions of the contours $\gamma_1^{(+)}$ and $\gamma_2^{(+)}$ go to zero as $R\rightarrow0$. Similarly, the contributions of $\gamma_1^{(-)}$ and $\gamma_2^{(-)}$ go to zero and as a result 
$\Gamma_R\to H$ , as $R\to\infty$, where $H$ is the Hankel contour in Figure \ref{fig2}. We first consider the contour $\gamma_1^{(+)}$ parametrized 
by $z(t)=c+\i t$, $t\in[R,\infty)$. For every $x>0$, we have
\begin{align*}
\left| \frac{1}{2\pi \i} \int_{\gamma_1^{(+)}} \phi(z)e^{zx}\d z\right|
& \leq \frac{1}{2\pi} \int_{R}^{\infty} \left| \phi(c+\i t)e^{(c+\i t)x}\right|\d t\\
& \leq \frac{1}{2\pi}e^{cx}M_2 \int_{R}^{\infty} |c+\i t|^{-2} \d t = \frac{1}{2\pi}e^{cx} M_2\int_{R}^{\infty} \frac{1}{c^2+t^2} \d t \longrightarrow 0, \text{ as } R\to \infty,
\end{align*}
where we have used lemma \ref{lemma1} with $k=2$. Next we consider the contour $-\gamma_2^{(+)}$ parametrized by 
$z(t)=c+Re^{\i t}$, $t\in[\frac{\pi}{2},\pi-\theta_R]$, where $\theta_R \to 0$ as $R\to\infty$. Since
\begin{equation*}
|z(t)|=\left|c+Re^{\i t}\right|\geq R-c = \left(1-\frac{c}{R}\right)R,
\end{equation*}
there exists $R'>0$ such that $|z(t)|\geq \frac{1}{2}R$ when $R>R'$. So for every $x>0$, and $R>R'$, we have 
\begin{align*}
\left| \frac{1}{2\pi i} \int_{\gamma_2^{(+)}} \varphi(z)e^{zx}\phantom{a}dz\right|
& \leq \max_{z\in-\gamma_2^{(+)}} \left\{ |\varphi(z)|e^{\re{(z)}x}\right\} \cdot \frac{\pi}{2} R\\
& \leq \max_{t\in[\frac{\pi}{2},\pi-\theta_R]} \left\{M_k|z(t)|^{-k}e^{(c+R\cos{t})x}\right\}\cdot\frac{\pi}{2} R\\
& \leq  M_k\left(\frac{1}{2}R\right)^{-k} e^{(c+R\cos{(\frac{\pi}{2})})x}\cdot\frac{\pi}{2} R = \pi2^{k-1} M_ke^{cx}R^{-k+1}\longrightarrow0, \text{ as } R\to\infty,
\end{align*}
when we use lemma 1 with any $k>1$. Thus, taking $R\to\infty$, we have
\begin{equation}\label{density_hankel}
f_X(x)= \frac{1}{2\pi \i} \int_{H}\phi(z)e^{zx} \d z.
\end{equation}

\begin{figure}
\centering
\begin{tikzpicture}
[decoration={markings,
mark=at position 1cm with {\arrow[line width=1pt]{>}},
mark=at position 4.1cm with {\arrow[line width=1pt]{>}},
mark=at position 7.45cm with {\arrow[line width=1pt]{>}}
}
]
\draw[help lines,->] (-3,0) -- (3,0) coordinate (xaxis);
\draw[help lines,->] (0,-2) -- (0,2) coordinate (yaxis);

\path[draw,line width=0.8pt,postaction=decorate] (-3,-0.5)-- (-0.5,-0.5) arc(-135:135:0.707)--(-3,0.5);

\draw[dashed,->] (-1.5,0)--(-1.5,0.5);
\draw[dashed,->] (0,0)--(-0.5,0.5);
\draw[-] (-0.25,0) arc(180:135:0.25);
\draw[dashed,->] (0,0)--(0.5,0.5);
\node[below] at (xaxis) {$\re{(z)}$};
\node[left] at (yaxis) {$\im{(z)}$};
\node[below right] {$0$};
\node at (1,0.3) {$H$};
\node at (-1.25,0.25) {\small{$\epsilon$}};
\node at (-0.45,0.15) {\small{$\theta_\epsilon$}};
\node at (0.2,0.45) {\small{$\delta$}};
\end{tikzpicture}
\caption{ The Hankel contour $H$}
\label{fig2}
\end{figure}

The contour $H$ is defined $\forall\delta>0$ and $\forall \epsilon\in(0,\delta)$ and it is clear that $\theta_\epsilon\to0$ as $\epsilon\to0$. Rewriting \eqref{density_hankel} as 
\begin{equation*}
f_X(x)= \frac{1}{2\pi \i} \left\{\left (\int_{h_1}+\int_{h_2}+\int_{h_3}\right)\phi(z)e^{zx} \d z\right\},
\end{equation*}
where,
\begin{align*}
& \int_{h_1} \phi(z)e^{zx} \d z = \int_{\delta\cos{\theta_\epsilon}}^{\infty} \phi(-t-\i\epsilon)e^{(-t-\i\epsilon)x} \d t,\\
& \int_{h_2} \phi(z)e^{zx} \d z = \i\delta\int_{-\pi + \theta_\epsilon}^{\pi - \theta_\epsilon} \phi(\delta e^{\i t})e^{\delta e^{\i t}x+\i t}\d t, \text{ and}\\
& \int_{h_3} \phi(z)e^{zx} \d z = -\int_{\delta\cos{\theta_\varepsilon}}^{\infty} \varphi(-t+\i\epsilon)e^{(-t+\i\epsilon)x} \d t,
\end{align*}
and using the fact that $\overline{\phi(z)}=\phi(\overline{z})$, we have 
\begin{align*}
f_X(x) &= \frac{1}{2\pi \i} \Bigg\{ \i\delta\int_{-\pi + \theta_\epsilon}^{\pi - \theta_\epsilon} \phi(\delta e^{\i t})e^{\delta e^{\i t}x+\i t}\d t
         -2\i \im{\left[\int_{\delta\cos{\theta_\epsilon}}^{\infty} \phi(-t+\i\epsilon)e^{(-t+\i\epsilon)x} \d t\right]}\Bigg\}\\
       & \longrightarrow -\frac{1}{\pi} \im{\left[\int_{0}^{\infty} \phi(-t+\i \cdot0)e^{-tx} \d t\right]}, \text{ as } \epsilon,\delta \to 0\\
       &=-\frac{1}{\pi} \int_{0}^{\infty} \im{[\phi(-t+\i\cdot0)]}e^{-tx} \d t.
\end{align*}
The interchange of the limit and integration can be justified by dominated convergence. 
\end{proof}

To utilize Proposition \ref{prop1}, one must first compute $\phi_j(-t+\i\cdot0)$, $j=1,\ldots,n$. This can be performed using the methods from Section \ref{section4}, or any alternative method (for example, see \cite{Bondesson2002}). Since the random variables, $X_j$, $j=1,\ldots,n$, are independent we have
\begin{equation*}
 \phi(-t+\i\cdot0)=\prod_{j=1}^{n}\phi_j(-t+\i\cdot0).
\end{equation*}
We can compute the integral in \eqref{density_integral} in a similar fashion to the numerical integration method of Section \ref{section4.2}.

To illustrate the method, we computed the Laplace transform of $X \sim \LN(0,1)$ using the Theorem \ref{thm3} and used the inversion formula of Proposition \ref{prop1} to obtain the density. Figure \ref{fig3a} shows a plot with both the closed form of $f_X$ and our approximation. Figure \ref{fig3b} shows the relative error of the approximation.

\begin{figure}
\centering
\begin{subfigure}{.5\textwidth}
  \centering
  \includegraphics[width=\linewidth]{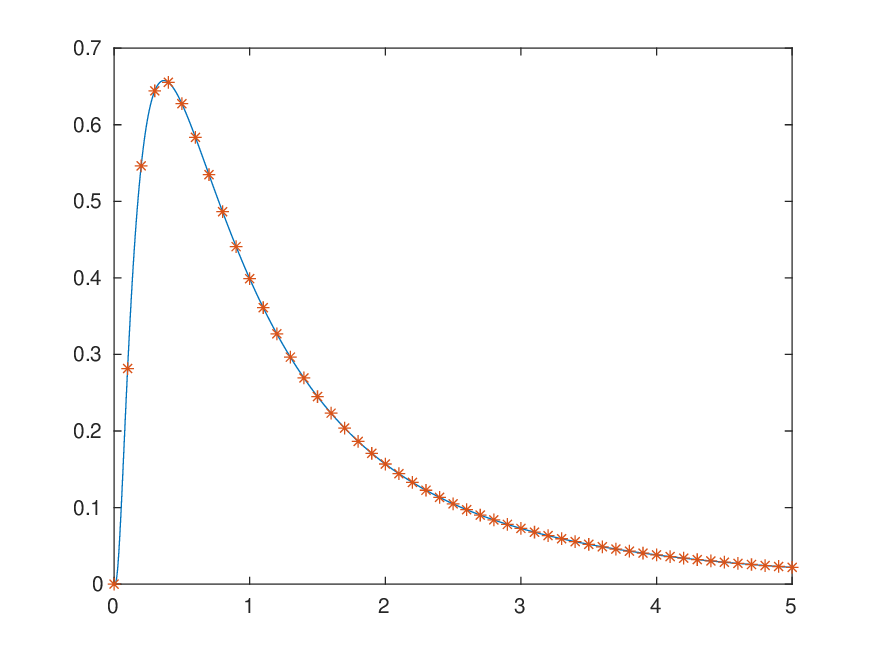}
  \caption{Actual density and our approximation.}
  \label{fig3a}
\end{subfigure}%
\begin{subfigure}{.5\textwidth}
  \centering
  \includegraphics[width=\linewidth]{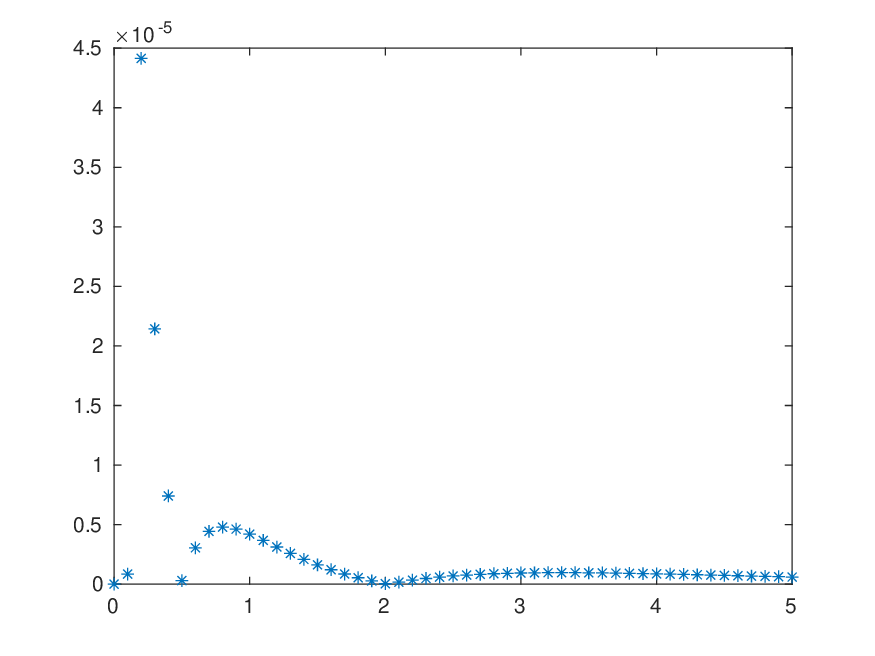}
  \caption{Relative error}
  \label{fig3b}
\end{subfigure}
\caption{Approximating the density of $X \sim \LN(0,1)$}
\label{Figure3}
\end{figure}


\section{Conclusion}\label{conclusion}

We have presented two derivations of the analytic continuation of the Laplace transform of the lognormal distribution, which we denote by $\phi$. Since the Mellin 
transform of $\phi$ has closed form, we used the Mellin inversion formula to express $\phi$ in the form of a Mellin-Barnes integral. As a consequence, we 
obtained the corresponding expression for the characteristic function of the lognormal distribution. This expression is slightly different from the expression derived 
by Leipnik in \cite{Leipnik1991}; we claim his expression is incorrect.

Using the Mellin-Barnes expression for $\phi$, we obtained two approximations which may be used in numerical computations. The error of the first approximation (see Theorem \ref{thm3})
can be made arbitrarily small and the approximation is asymptotic to $\phi$ as the magnitude of the argument goes to zero. 
The second approximation (see Theorem \ref{thm4}) improves as the parameter $\sigma$ goes to infinity. Both approximations were shown to provide accurate results, however, we note that computation can be difficult if too many terms of the series employed.

In the last section, we showed how one may use the analytic continuation of the Laplace transform of a sum of independent lognormals to compute the density, via 
Laplace inversion. By deforming the vertical contour of the Bromwich integral to a Hankel contour, one may obtain a real integral for which the integrand decays 
exponentially. The result is an integral that can be computed numerically with ease.

The analytic continuation of the Laplace transform of the lognormal distribution has other applications. In 1977, Olof Thorin showed that the lognormal distribution is a Generalized Gamma Convolution
(GGC) (see \cite{Thorin1977}). A GGC is a probability distribution $F$ on $[0,\infty)$ with moment-generating function (mgf) of the form 
\begin{equation*}
 M(s)= \exp{\left[as+\int_{0}^{\infty}\ln{\left(\frac{t}{t-s}\right)} U(\d t)\right]}, \spc s\leq0 \text{ (or }s\in\C
 \setminus(0,\infty)\text{)},
\end{equation*}
where $a\geq0$ and $U(\d t)$ is a nonnegative measure, called the Thorin measure, on $(0,\infty)$ satisfying
\begin{equation*}
 \int_{(0,1]} | \ln{t}| U(\d t)<\infty \text{, and } \int_{(1,\infty)} t^{-1}U(\d t)<\infty,
\end{equation*}
(\cite{Bondesson1992}, pg. 29). As Bondesson discusses in \cite{Bondesson1992} and \cite{Bondesson2002}, one may compute the density of the Thorin measure using the
analytic continuation of the Laplace transform of the lognormal distribution. The density, denoted here by U, can be computed using the formula
\begin{equation*}
 U(t) = \frac{1}{\pi}\text{Im}\left[\frac{\phi'(-t+\i\cdot0)}{\phi(-t+\i\cdot0)}\right],
\end{equation*}
where $\phi(-t +\i\cdot0)=\lim_{\epsilon\to0^+} \phi(-t+\i\epsilon)$ (equivalently, as in Bondesson's derivation, one may approach the negative real line from 
below and multiply the result by -1).

\section*{Acknowledgment}
I would like to express my gratitude to Alexey Kuznetsov for his guidance and advice.


\begin{thebibliography}{10}

\bibitem{Asmussen2016}
S{\o}ren Asmussen, Jens~Ledet Jensen, and Leonardo Rojas-Nandayapa.
\newblock On the {L}aplace transform of the lognormal distribution.
\newblock {\em Methodology and Computing in Applied Probability},
  18(2):441--458, Jun 2016.

\bibitem{Barakat1976}
Richard Barakat.
\newblock Sums of independent lognormally distributed random variables.
\newblock {\em J. Opt. Soc. Am.}, 66(3):211--216, Mar 1976.

\bibitem{Barouch1976}
E.~Barouch and Gordon~M. Kaufman.
\newblock On sums of lognormal random variables.
\newblock Working paper, Alfred P. Sloan School of Management, MIT.

\bibitem{Bondesson1992}
Lennart Bondesson.
\newblock {\em Generalized Gamma Convolutions and Related Classes of
  Distributions and Densities}, volume~76.
\newblock Springer-Verlag New York, 1st edition, 1992.

\bibitem{Bondesson2002}
Lennart Bondesson.
\newblock On the {L}{\'e}vy measure of the lognormal and the log{C}auchy
  distributions.
\newblock {\em Methodology And Computing In Applied Probability},
  4(3):243--256, Sep 2002.

\bibitem{Dufresne2008}
Daniel Dufresne.
\newblock Sums of lognormals.
\newblock In {\em Actuarial Research Conference Proceedings}, 2008.

\bibitem{Filon1928}
L.~N.~G. Filon.
\newblock On a quadrature formula for trigonometric integrals.
\newblock {\em Proceedings of the Royal Society of Edinburgh}, 49:38–47,
  1928.

\bibitem{Gubner2006}
John~A. Gubner.
\newblock A new formula for lognormal characteristic functions.
\newblock {\em Vehicular Technology, IEEE Transactions on}, 55(5):1668 -- 1671,
  Oct 2006.

\bibitem{Holgate1989}
P.~Holgate.
\newblock The lognormal characteristic function.
\newblock {\em Communications in Statistics - Theory and Methods},
  18:4539--4548, Jan 1989.

\bibitem{Laub2016}
Patrick~J. Laub, Søren Asmussen, Jens~L. Jensen, and Leonardo Rojas-Nandayapa.
\newblock Approximating the {L}aplace transform of the sum of dependent
  lognormals.
\newblock {\em Advances in Applied Probability}, 48(A):203–215, 2016.

\bibitem{Leipnik1991}
Roy~B. Leipnik.
\newblock On lognormal random variables: I-the characteristic function.
\newblock {\em The Journal of the Australian Mathematical Society. Series B.
  Applied Mathematics}, 32(3):327–347, 1991.

\bibitem{Tellambura2010}
C.~Tellambura and D.~Senaratne.
\newblock Accurate computation of the mgf of the lognormal distribution and its
  application to sum of lognormals.
\newblock {\em Trans. Comm.}, 58(5):1568--1577, May 2010.

\bibitem{Thorin1977}
Olof Thorin.
\newblock On the infinite divisibility of the lognormal distribution.
\newblock {\em Scandinavian Actuarial Journal}, 1977(3):121--148, 1977.

\bibitem{Titchmarsh}
E.~C. Titchmarsh.
\newblock {\em Introduction to the Theory of Fourier Integrals}.
\newblock Oxford University Press, 2nd edition, 1948.

\end{thebibliography}

\end{document}